\documentclass[11pt,twoside,english]{amsart}
\normalsize
\advance\oddsidemargin by -1.0cm
\advance\evensidemargin by -1.0cm
\advance\topmargin by -1.0cm 
\usepackage{amssymb}
\usepackage{babel}
\usepackage{amsmath}
\usepackage{amscd}   
\usepackage{epsfig}  
\usepackage{rotating}
\usepackage{psfrag}
\let\mathg\mathfrak
\theoremstyle{plain}
\newtheorem{cor}{Corollary}[section]
\newtheorem{lem}{Lemma}[section]
\newtheorem{thm}{Theorem}[section]
\newtheorem{prop}{Proposition}[section]

\theoremstyle{definition}

\newcommand{\bdm}{\begin{displaymath}}
\newcommand{\edm}{\end{displaymath}}
\newcommand{\be}{\begin{equation}}
\newcommand{\ee}{\end{equation}}
\newcommand{\ba}[1]{\begin{array}{#1}}
\newcommand{\ea}{\end{array}}
\newcommand{\bea}[1][]{\begin{eqnarray#1}}
\newcommand{\eea}[1][]{\end{eqnarray#1}}
\newcommand{\btab}{\begin{tabular}}
\newcommand{\etab}{\end{tabular}}

\newcommand{\ra}{\rightarrow}

\newcommand{\cyclic}[1]{\stackrel{{\scriptsize #1}}{\mathfrak{S}}}
\newcommand{\C}{\ensuremath{\mathbb{C}}}
\newcommand{\R}{\ensuremath{\mathbb{R}}}

\newcommand{\M}{\ensuremath{\mathcal{M}}}

\newcommand{\End}{\ensuremath{\mathrm{End}}}

\newcommand{\kr}{\ensuremath{\mathcal{R}}}
\newcommand{\Ric}{\ensuremath{\mathrm{Ric}}}
\newcommand{\Scal}{\ensuremath{\mathrm{Scal}}}

\newcommand{\so}{\ensuremath{\mathg{so}}}
\newcommand{\SO}{\ensuremath{\mathrm{SO}}}
\newcommand{\Spin}{\ensuremath{\mathrm{Spin}}}

\newcommand{\g}{\ensuremath{\mathfrak{g}}}

\newcommand{\X}{\ensuremath{\mathfrak{X}}}
\begin{document}
\def\haken{\mathbin{\hbox to 6pt{%
                 \vrule height0.4pt width5pt depth0pt
                 \kern-.4pt
                 \vrule height6pt width0.4pt depth0pt\hss}}}
    \let \hook\intprod
\setcounter{equation}{0}
%
%
\thispagestyle{empty}
%
\date{\today}
\title[Flat metric connections with antisymmetric torsion]{A note on flat 
metric connections with antisymmetric torsion }
%
%
%
\author{Ilka Agricola}\author{Thomas Friedrich}
\address{\hspace{-5mm} 
Ilka Agricola\newline
Fachbereich Mathematik und Informatik \newline
Philipps-Universit\"at Marburg\newline
Hans-Meerwein-Strasse \newline
D-35032 Marburg, Germany\newline
{\normalfont\ttfamily agricola@mathematik.uni-marburg.de}}
\address{\hspace{-5mm} 
Thomas Friedrich\newline
Institut f\"ur Mathematik \newline
Humboldt-Universit\"at zu Berlin\newline
Sitz: WBC Adlershof\newline
D-10099 Berlin, Germany\newline
{\normalfont\ttfamily friedric@mathematik.hu-berlin.de}}
%
%
\thanks{Supported by the Junior Research Group "Special Geometries in Mathematical Physics"
of the Volkswagen\textbf{Foundation}.}
\subjclass[2000]{Primary 53 C 25; Secondary 81 T 30}
\keywords{flat connections, 
skew-symmetric torsion}  
\begin{abstract}
In this short note we study flat metric
connections with antisymmetric torsion $T \neq 0$. 
The result has been originally
discovered by Cartan/Schouten in 1926 and we provide a new proof not
depending on the classification of symmetric spaces. 
Any space of that type splits and the
irreducible factors are  compact simple Lie group
or a special connection on $S^7$. The latter case is interesting from the
viewpoint of $G_2$-structures and we discuss its type in the sense of the
Fernandez-Gray classification. Moreover, we investigate flat metric
connections of vectorial type. 
\end{abstract}
\maketitle
\pagestyle{headings}
%
%
%
\section{Introduction}\noindent
\noindent
Consider a complete Riemannian manifold $(M^n,g, \nabla)$ endowed with a 
metric connection $\nabla$. The torsion $T$ of  $\nabla$, viewed as a 
$(3,0)$ tensor, is defined by
\bdm
T(X,Y,Z)\ :=\ g(T(X,Y),Z)\ =\ g(\nabla_XY-\nabla_YX-[X,Y] ,Z).
\edm
Metric connections for which $T$ is antisymmetric in all arguments,
i.\,e.~$T\in\Lambda^3(M^n)$ are of particular interest, see
\cite{Agricola06}.
They correspond precisely to those metric connections that have the same 
geodesics as the Levi-Civita connection. In this note we will investigate 
\emph{flat} connections of that type.

The observation that any simple Lie group carries in fact two flat 
connections, 
usually called the $(+)$- and the $(-)$-connection, with 
torsion $T(X,Y)=\pm [X,Y]$ is due to \'E.~Cartan and J.\,A.~Schouten \cite{Cartan&Sch26a}
and is explained in detail in \cite[p.~198-199]{Kobayashi&N2}. If one then chooses a
biinvariant metric, these connections are metric and the torsion becomes
a $3$-form, as desired. Hence, the question is whether there are any further
examples of flat metric connections with antisymmetric torsion beside products 
of Lie groups.

The answer can be found in Cartan's work. 
In fact, \'E.~Cartan and J.\,A.~Schouten
published a second joint paper very shortly after the one mentioned above, 
\cite{Cartan&Sch26b}.
There is only one additional such geometry,  realized on $S^7$. Their proof that no more cases can occur is
by diligent 
inspection of some defining tensor fields.

Motivated from the problem when the Laplacian of a Riemannian manifold can
at least locally be written as  a sum $\Delta = -\sum X_i\circ X_i$,
d'Atri and Nickerson investigated in 1968 manifolds which admit
an orthonormal frame consisting of Killing vector fields. This question is 
almost equivalent to the previous.
In two beautiful papers, Joe Wolf picked 
up the question again in the early 70ies and provided a complete 
classification of 
all  complete (reductive) pseudo-Riemannian manifolds admitting absolute 
parallelism, thus reproving the Cartan-Schouten result by other means 
\cite{Wolf72a}, \cite{Wolf72b}. The key observation was that the
Riemannian curvature of such a space must, for three $\nabla$-parallel
vector fields, be given by $R(X,Y)Z=- [[X,Y],Z]/4$, and thus defines a 
Lie triple system. The proof is then reduced to an (intricate) algebraic
problem about Lie triple systems, and $S^7$ (together with two
pseudo-Riemannian siblings) appears because of the outer automorphism
inherited from triality.\\
 
\noindent
The main topic of this paper is to understand this very interesting result
in terms of special geometries with torsion. We will give a new and elementary
proof of the
result not using the classification of symmetric spaces. Moreover, we 
describe explicitely
the family of flat metric connections with antisymmetric torsion
on $S^7$ and  make the link to
$G_2$ geometry apparent. 
%
\section{The case of skew symmetric torsion}\noindent
%
Let $(M^n,g,\nabla)$ be a connected Riemannian
manifold endowed with a flat metric connection $\nabla$. The 
parallel transport of any orthonormal 
frame in a point will define a local orthonormal frame  $e_1,\ldots.e_n$
in all other points. In the sequel, no distinction will be
made between vector fields and $1$-forms. The standard formula for the
exterior derivative of a $1$-form yields
\begin{eqnarray*}
d e_i(e_j,e_k)& =& e_j\langle e_i,e_k\rangle - e_k\langle e_i,e_j\rangle
-\langle e_i,[e_j,e_k]\rangle\ =\ -\langle e_i,[e_j,e_k]\rangle \notag \\
&=& -\langle e_i, \nabla_{e_j} e_k -\nabla_{e_k}e_j - T(e_j,e_k)\rangle
\ =\  \langle e_i, T(e_j,e_k)\rangle. \label{dei}
\end{eqnarray*}
Hence, the torsion can be computed from  the frame
$e_i$ and their differentials.
As was shown by Cartan 1925, the torsion $T$ of $\nabla$ can basically be 
of $3$ possible types---a $3$-form, a vector, and a more difficult type
that has no geometric interpretation \cite{Tricerri&V1}, \cite{Agricola06}.
We shall first study the case that the torsion is a $3$-form. We state the explicit formula for the torsion and
draw some first conclusions from the identities relating the curvatures
of $\nabla$ and $\nabla^g$, the Levi-Civita connection.
Let the flat connection $\nabla$ be given by
\bdm
\nabla_X Y\ =\ \nabla^g_X Y +\frac{1}{2}T(X,Y,-) 
\edm
for a $3$-form $T$. 
The general relation between $\Ric^g$ and $\Ric^\nabla$
\cite[Thm A.1]{Agricola06} yields for any orthonormal frame $e_1,\ldots,e_n$
that the Riemannian Ricci tensor can be computed directly from $T$, 
\bdm
\Ric^g(X,Y)\ =\ \frac{1}{4}\sum_{i=1}^n \langle T(X,e_i),T(Y,e_i) \rangle, \quad
\Scal^g\ = \ \frac{3}{2}\|T\|^2.
\edm
In particular, $\Ric^g$ is non-negative, $\Ric^g(X,X)\geq 0$ for all $X$,
and $\Ric^g(X,X)=0$ if and only if $X\haken T=0$.
The torsion form $T$ is coclosed, $\delta T=0$,  
because it coincides with the skew-symmetric part of  $\Ric^\nabla = 0$. We
define the $4$-form $\sigma_T$ by the formula
\bdm
\sigma_T \ := \, \frac{1}{2} \sum_{i=1}^n ( e_i \haken T) \wedge (e_i \haken
T) \ ,
\edm 
or equivalently by the formula
\bdm
\sigma_T(X,Y,Z,V)\ :=\ \langle T(X,Y),T(Z,V)\rangle
+\langle T(Y,Z),T(X,V)\rangle
+\langle T(Z,X),T(Y,V)\rangle \ .
\edm
Denote by $\nabla^{1/3}$  the metric connection with torsion $T/3$. Then we
can formulate some properties of the Riemannian manifold and the torsion form.

\begin{prop}
Let $\nabla$ be a flat metric connection with torsion
$T\in\Lambda^3(M^n)$. Then
\bdm
3 \, dT \ = \ 2 \, \sigma_T \, , \quad \nabla^{1/3} T \ = \ 0 \, , 
\quad \nabla^{1/3} \sigma_T \ = \ 0 \ .
\edm
The covariant derivative $\nabla T$ is a $4$-form and given by
\bdm
(\nabla_V T)(X,Y,Z)\ =\ \frac{1}{3}\sigma_T(X,Y,Z,V) \, \quad
\mbox{or} \quad \nabla_V T \ = \ - \, \frac{1}{3} \, (V \haken
\sigma_T) \ ,
\edm
\bdm
(\nabla^g_V T)(X,Y,Z)\ =\ - \, \frac{1}{6}\sigma_T(X,Y,Z,V) \, \quad
\mbox{or} \quad \nabla^g_V T \ = \  \frac{1}{6} \, (V \haken
\sigma_T) \ .
\edm
In particular, the length $||T||$
and the scalar curvature are constant.
The full Riemann curvature tensor is given by
\bdm
\kr^g(X,Y,Z,V)\ =\ - \frac{1}{6}\langle T(X,Y),T(Z,V)\rangle
+\frac{1}{12}\langle T(Y,Z),T(X,V)\rangle
+\frac{1}{12}\langle T(Z,X),T(Y,V)\rangle,
\edm
and  is $\nabla^{1/3}$-parallel, $\nabla^{1/3} \kr^g = 0$. Finally,
the sectional curvature is non-negative,
\bdm
K(X,Y)\ =\ \frac{\|T(X,Y)\|^2}{4[\|X\|^2\|Y\|^2-\langle X,Y\rangle^2]}\ \geq\ 0.
\edm
\end{prop}
\begin{proof}
The first Bianchi identity \cite[Thm 2.6]{Agricola06},
\cite{Friedrich&I1} states for flat $\nabla$
\be\label{Bianchi-I}
dT(X,Y,Z,V) - \sigma_T(X,Y,Z,V)+(\nabla_V T)(X,Y,Z)\ =\ 0  .
\ee
By the general formula
\cite[Cor.~3.2]{Friedrich&I1} we have $3dT = 2 \, \sigma_T $
for any flat connection with skew-symmetric torsion.
Together with equation ($\ref{Bianchi-I}$), 
this shows the first and second formula. 
The expression for the curvature follows from this and 
the general identity \cite[Thm A.1]{Agricola06}, \cite{Friedrich&I1}
\bea[*]
\kr^g(X,Y,Z,V)& =& \kr^\nabla(X,Y,Z,V) -\frac{1}{2}(\nabla_X T)(Y,Z,V)\\
&& +\frac{1}{2}(\nabla_Y T)(X,Z,V)-\frac{1}{4}\langle T(X,Y),T(Z,V)\rangle
-\frac{1}{4}\sigma_T(X,Y,Z,V).
\eea[*]
Since $\nabla  -  \nabla^{1/3} = \frac{1}{3}  T$, we obtain
\bdm
(\nabla_VT)(X,Y,Z)  -  (\nabla^{1/3}_V T)(X,Y,Z) \ = \ 
\frac{1}{3} \, T(V , - ,  - ) [T] (X,Y,Z)  , 
\edm
where $T(V , - ,  - ) [T]$ denotes the action of the $2$-form
$T(V , - ,  - )$ on the $3$-form $T$. Computing this action, we
obtain
\bdm
 T(V , -  ,  - ) [T] (X,Y,Z) \ = \ \sigma_T(X,Y,Z,V)  .
\edm
$\nabla^{1/3} T = 0$ follows now directly from the formula for $\nabla T$. 
In a similar way we compute $\nabla^gT$,
\bdm
\nabla^g_V T \ = \ \nabla_V T -  \frac{1}{2} \, (V \haken T)[T] \ = \ 
- \, \frac{1}{3} \, (V \haken \sigma_T)  +  \frac{1}{2} \,
(V \haken \sigma_T) \, = \, \frac{1}{6} (V \haken \sigma_T)  .\qedhere
\edm
\end{proof}
\noindent
Observe that the curvature identity of the last proposition is
nothing than formula (6) in \cite{Cartan&Sch26b} and the formula for
$\nabla^gT$ is formula (7) in the Cartan/Schouten paper. 
\begin{cor}\label{parallel-T-pol}
Consider a tensor field $\mathcal{T}$ being a polynomial of the torsion form
$T$. Then we have
\bdm
\nabla \mathcal{T} \ = \ - \, 2 \, \nabla^g \mathcal{T} \ .
\edm 
In particular, $\mathcal{T}$ is $\nabla$-parallel if and only if it is $\nabla^g$-parallel.
\end{cor}
\noindent
We derive the following splitting principle, which can again be found in
\cite{Cartan&Sch26b}.
\begin{prop}
If $M^n = M_1^{n_1} \times M_2^{n_2}$ is the Riemannian product and $T$ is the
torsion form of a flat metric connection, then T splits into $T = T_1 + T_2$,
where $T_i \in \Lambda^3 (M_i^{n_i})$ are $3$-forms on $M_i^{n_i}$. Moreover,
the connection splits,
\bdm
(M^n , g , \nabla) \ = \ (M^{n_1} , g_1 , \nabla^1) \, \times \, 
(M^{n_2} , g_2 , \nabla^2)
\edm 
\end{prop}
\begin{proof}
Consider two vectors $X \in T(M^{n_1}) , \, Y \in T(M^{n_2})$. Then the sectional
curvature of the $\{X,Y\}$-plane vanishes,
$\kr^g(X,Y,Y,X) = 0$. Consequently, we conclude that
$X \haken (Y \haken T) = 0$ holds.
\end{proof}
\noindent
In the simply connected and complete case we can decompose our flat metric
structure into a product of irreducible ones (de Rham decomposition
Theorem). Consequently we assume from now on that $M^n$ is a complete,
simply connected and
irreducible Riemannian manifold and that $T \neq 0$ is non-trivial. The
$\nabla$-parallel vector fields $e_1 , \ldots , e_n$ are Killing 
and we immediately obtain the formulas
\bdm
\nabla^g_{e_k} e_l \ = \ -  \nabla^g_{e_l} e_k  , \quad 
[e_k , e_l] \ = \ 2 \, \nabla^g_{e_k} e_l \ = \ -  T(e_k, e_l) 
\edm
and
\bdm
e_k \big( \langle [e_i,e_j], e_l \rangle \big) \ = \ -  
(\nabla_{e_k}T)(e_i,e_j,e_l) \ =
\  - \,\frac{1}{3} \, \sigma_T (e_i , e_j , e_l , e_k) .
\edm
In particular, $e_k \big( \langle [e_i,e_j], e_l \rangle \big)$ 
is totally skew-symmetric
and the function $ \langle [e_i,e_j], e_l \rangle$ is constant if and only 
if the torsion form is $\nabla$-parallel, $\nabla T = 0$ (see 
\cite{DAtri&N68}, Lemma 3.3 and Proposition 3.7).
\begin{prop}[{see \cite{DAtri&N68}, Lemma 3.4}]
The Riemannian curvature tensor in the frame $e_1, \ldots , e_n$ is given by
the formula
\bdm
 \kr^g ( e_i , e_j) e_k\ =  - \, \frac{1}{4} \big[ [ e_i, e_j ] , e_k \big] .
\edm
In particular, $ \kr^g ( e_i , e_j ) e_k$ is a Killing vector field.
\end{prop}
\begin{proof}
We compute
\begin{eqnarray*}
\langle [e_i,e_j],[e_k,e_l] \rangle &=& 2 \, \langle [e_i,e_j] ,
\nabla^g_{e_k} e_l \rangle \ = \ 
- \, 2 \, \langle \nabla^g_{e_k} [e_i,e_j] , e_l \rangle  +  2  e_k \big( 
\langle [e_i,e_j] , e_l \rangle \big) \\
&=& -  2  \, \big\langle \nabla^g_{[e_i,e_j]}e_k + \big[ e_k , [e_i ,e_j] \big] ,
e_l \big\rangle  +  2  \, e_k \big( \langle [e_i,e_j] , e_l \rangle \big) \\
&=& 2 \, \langle  \nabla^g_{e_l} e_k , [e_i, e_j] \rangle  +  
2 \, \langle \big[ [e_i,e_j] , e_k \big] , e_l \rangle  +   
2 \, e_k \big( \langle [e_i,e_j] , e_l \rangle \big) \\
&=& \langle [e_l, e_k],[e_i,e_j] \rangle  + 
2 \, \langle \big[ [e_i,e_j] , e_k \big] , e_l \rangle  +   
2 \, e_k \big( \langle [e_i,e_j] , e_l \rangle \big)  
\end{eqnarray*}
and we obtain the following formula
\bdm
\langle [e_i,e_j],[e_k,e_l] \rangle \ = \  \langle \big[ [e_i,e_j],e_k \big] , 
e_l \rangle +  
e_k \big( \langle [e_i,e_j] , e_l \rangle \big)  .
\edm
The required formula follows now from the Jacobi identity and the fact, that 
$e_k \big( \langle [e_i,e_j] , e_l \rangle \big)$ is totally skew-symmetric,
\begin{eqnarray*}
\kr^g(e_i,e_j, e_k,e_l) &=& - \, \frac{1}{6}  \langle [e_i,e_j],[e_k,e_l]
\rangle  +
\,  \frac{1}{12} \, \langle [e_j,e_k],[e_i,e_l] \rangle +  
\frac{1}{12} \, \langle [e_k,e_i],[e_j,e_l] \rangle \\
&=&  - \, \frac{1}{6} \, \langle \big[[e_i,e_j],e_k\big] ,e_l \rangle -  
\frac{1}{6} \,  e_k \big( \langle [e_i,e_j] , e_l \rangle \big) 
 + \,  \frac{1}{12} \, \langle \big[[e_j,e_k],e_i\big] ,e_l \rangle\\
&& + \frac{1}{12} \,  e_i \big( \langle [e_j,e_k] , e_l \rangle \big) 
 + \,  \frac{1}{12} \, \langle \big[[e_k,e_i],e_j\big] ,e_l \rangle +  
\frac{1}{12} \,  e_j \big( \langle [e_k,e_i] , e_l \rangle \big) \\
&=& - \, \frac{1}{4} \, \langle \big[[e_i,e_j],e_k \big] , e_l \rangle . \hspace{7cm} 
\qedhere
\end{eqnarray*}
\end{proof}
\begin{lem}
Let $X,Y$ be a pair of Killing vector fields such that 
$\langle X,Y \rangle$ is constant and let $Z$ be a third Killing
vector field. Then 
\bdm
X \big( \langle Y , Z \rangle \big) \ = \ - \, Y \big( \langle X , Z \rangle 
\big) .
\edm
is skew-symmetric in $X,Y$.
\end{lem}
\begin{proof}
For any vector field $W$, we obtain
\bdm
\langle \nabla^g_X Y , W \rangle \ = \ -  \langle \nabla^g_W Y,X \rangle \ = \ 
- \, W \big(\langle X,Y \rangle \big)  +  \langle  Y , \nabla^g_WX \rangle \ 
= \  - \, \langle \nabla^g_Y X , W \rangle  ,
\edm
i.\,e., $\nabla^g_X Y = - \nabla^g_Y X$. Then the result follows,
\bdm
X\big( \langle Y , Z \rangle \big) \, = \, \langle \nabla^g_XY , Z \rangle  
+  \langle Y , \nabla^g_XZ \rangle \, 
= \, - \, \langle \nabla^g_Y X , Z \rangle  - \langle X , \nabla^g_Y Z \rangle
\, =\, - 
Y \big( \langle X , Z \rangle \big)  . \qedhere
\edm
\end{proof}
\noindent
Denote by $R_{ijkl} = \kr^g(e_i,e_j,e_k,e_l)$ the coefficients of the
Riemannian curvature with respect to the $\nabla$-parallel frame $e_1 , \ldots
, e_n$. Since $\big[[e_i,e_j],e_k\big]$ is a Killing vector field, the latter 
Lemma reads as $e_m(R_{ijkl}) =  - \, e_l(R_{ijkm})$ .
If $m$ is one of the indices $i,j,k,l$, we obtain $e_m(R_{ijkl}) = 0$
immediately. Otherwise we use in addition 
the symmetry properties of the curvature tensor,
\begin{eqnarray*}
e_1(R_{2345}) &=& - \, e_5(R_{2341}) \ = \ - \, e_5(R_{1432}) \ = \ 
e_2(R_{1435}) \ = \ - \, e_2(R_{3541})  \\
&=&   e_1(R_{3542}) \ =\ e_1(R_{4235}).
\end{eqnarray*}
Similarly one derives $e_1(R_{4235}) = e_1(R_{3425})$
and the Bianchi identity $R_{2345} + R_{4235} + R_{3425} = 0$ yields the
result, $e_1(R_{2345})= 0$. Consequently, the coefficients are constant and we
proved the following
\begin{thm}[{see \cite{Cartan&Sch26b}, formula (25), \cite{DAtri&N68}},
  Theorem 3.6]\label{curv-parallel}
The Riemannian curvature tensor $\kr^g$ is $\nabla$- and $\nabla^g$-parallel.
In particular, 
\bdm
\big[ \, X \haken T \, , \, \kr^g \, \big] \ = \ 0
\edm
holds for any vector $X \in T(M^n)$.
\end{thm}
\begin{proof}
$\kr^g$ is a polynomial depending on $T$. Consequently, $\nabla^g \kr^g = 0$
implies $\nabla \kr^g = 0$, see Corollary \ref{parallel-T-pol}. Hence,
the difference
\bdm
0 \ = \ (\nabla_X \, - \, \nabla^g_X) \kr^g \ = \ 
\big[ \, X \haken T \, , \, \kr^g \, \big]  
\edm
vanishes, too.
\end{proof}
\begin{cor}
Let $(M^n,g , \nabla, T)$ be a simply connected, complete and irreducible
Riemannian manifold equipped with a flat metric connection and totally
skew-symmetric torsion $T \neq 0$. Then $M^n$ is a compact, irreducible
symmetric space. Its Ricci tensor is given by
\bdm
\Ric^g(X,Y)\ =\ \frac{1}{4}\sum_{i=1}^n \langle T(X,e_i),T(Y,e_i) \rangle \ = 
\ \frac{\Scal^g}{n} \, \langle X
, Y \rangle \ , \quad
\Scal^g\ = \ \frac{3}{2}\|T\|^2  .
\edm
\end{cor}
\noindent
Since $\sigma_T$ is $\nabla^{1/3}$-parallel, there are two cases. If $\sigma_T
\equiv 0$, then the scalar products $\langle [e_i,e_j],e_k \rangle$ are constant, i.e.,  
the vector fields $e_1, , \ldots , e_n$ are a basis of a $n$-dimensional Lie
algebra. The corresponding simply connected Lie group is a simple, compact Lie
group and isometric to $M^n$. The torsion form of the flat connection is
defined by $T(e_k , e_l) = - \, [e_k , e_l]$ (see \cite{Kobayashi&N2}, chapter
X). \\

\noindent
The case $\sigma_T \not\equiv 0$ is more complicated. Since $\sigma_T$ is
a $4$-form, the dimension of the manifold is at least four. 
Cartan/Schouten (1926) proved that only the 
$7$-dimensional round sphere is possible. A different argument has 
been used by
D'Atri/Nickerson (1968) and Wolf (1972),  namely the
classification of irreducible, compact symmetric spaces with vanishing Euler 
characteristic. This list is very short. Except the compact, simple Lie groups
most of them do not admit Killing vector field of constant length. \\

\noindent
We provide now a new proof that does not use the
classification of symmetric spaces. 
Consider, at any point $m \in M^n$, the Lie algebra
\bdm
\hat{\g}_T(m) \ := \ Lie \big\{\, X \haken T \, : \ X \in T_m(M^n) \,
\big\} \ \subset \ \so(T_m(M^n))  ,
\edm 
that was introduced in \cite{Agri&F04} for the systematic investigation of
algebraic holonomy algebras. Since $T$ is $\nabla^{1/3}$-parallel, the algebras
$\hat{\g}_T(m)$ are $\nabla^{1/3}$-parallel, too.
\begin{prop}
Let $(M^n,g , \nabla, T)$ be a simply connected, complete and irreducible
Riemannian manifold equipped with a flat metric connection and totally
skew-symmetric torsion $T \neq 0$. Then the representation $(\hat{\g}_T(m) ,
T_m(M^n))$ is irreducible.
\end{prop}
\begin{proof}
Suppose that the tangent space splits at some point. Then any tangent space
splits and we obtain a $\nabla^{1/3}$-parallel decomposition $T(M^n) = V_1
\oplus V_2$ of the tangent bundle into two subbundles. Moreover, the torsion
form $T = T_1 \, + \, T_2$ splits into $\nabla^{1/3}$-parallel forms
$T_1 \in \Lambda^3(V_1)$ and  $T_2 \in \Lambda^3(V_2)$, see \cite{Agri&F04}.
The subbundles $V_1 , V_2$ are involutive and their
leaves are totally geodesic submanifolds of $(M^n,g)$. This contradicts the
assumption that $M^n$ is an irreducible Riemannian manifold.
\end{proof}
\noindent
If the Lie algebra $\hat{\g}_T \subset \so(n)$ of a $3$-form acts irreducibly
on the euclidian space, then there are two possibilities. 
Either the $3$-form of the euclidian space satisfies the Jacobi
identity or  the Lie algebra coincides with the full algebra, 
$\hat{\g}_T = \so(n)$ (see \cite{Agri&F04}, \cite{Nagy07}, \cite{OR08}). 
The first case again yields the result that the manifold $M^n$ is a
simple Lie group (we recover the case of $\sigma_T = 0$). Otherwise the Lie
algebra $\hat{g}_T$ coincides with $\so(n)$ and Theorem \ref{curv-parallel} 
implies that $M^n$
is a space of positive constant curvature,
$\kr^g =  c \cdot \mathrm{Id}$.
The formula for the sectional curvature
\bdm
K \ = \ K(X,Y)\ =\ \frac{\|T(X,Y)\|^2}{4[\|X\|^2\|Y\|^2-\langle X,Y\rangle^2]}
\edm
means that the $3$-form $T$ defines a metric vector cross product. 
Consequently, the dimension of the sphere is seven.
\begin{thm} [{see \cite{Cartan&Sch26b}}]
Let $(M^n,g , \nabla, T)$ be a simply connected, complete and irreducible
Riemannian manifold equipped with a flat metric connection and totally
skew-symmetric torsion $T \neq 0$. If $\sigma_T = 0$, then $M^n$ is isometric
to a compact simple Lie group. Otherwise ($\sigma_T \neq 0$)  $M^n$ is
isometric to $S^7$. 
\end{thm}

\section{The case of vectorial torsion}
%
%
\noindent
By definition, such a connection $\nabla$ is given by 
\bdm
\nabla_X Y\ =\ \nabla^g_X Y + \langle X,Y \rangle V - \langle V,Y \rangle X
\edm
for some vector field $V$. The general relation between  the curvature
transformations for $\nabla$ and $\nabla^g$ 
\cite[App. B, proof of Thm 2.6(1)]{Agricola06} reduces to
\bdm
\kr^g(X,Y)Z\ =\ \langle X,Z\rangle\nabla_Y V 
-\langle Y,Z\rangle\nabla_X V  
+ Y\langle \nabla_X V+ \|V\|^2 X,Z\rangle
- X\langle \nabla_Y V+ \|V\|^2 Y,Z\rangle.
\edm
Hence, the curvature depends not only on $V$, but also on $\nabla V$.
This remains true when considering the Ricci tensor, which does not
simplify much. However, the following claim may be read off
immediately: If $\nabla V=0$, then $M^n$ is a {\it non-compact} 
space of constant
negative sectional curvature $ - \, \|V \|^2$ and the divergence of the vector
field $V$ is constant, $\delta^g(V) = (n - 1 ) \|V
\|^2 = \mathrm{const} > 0$. Moreover, the integral curves of the vector field
$V$ are geodesics in $M^n$, $\nabla^g_V V = 0$. In 
\cite{Tricerri&V1}, this case is discussed in detail; in particular,
a flat metric connection with vectorial torsion is explicitely
constructed. \\

\noindent
For the general case ($\nabla V\neq 0$), the first Bianchi identity
\cite[Thm 2.6]{Agricola06} for a flat connection
\bdm
0 \ = \ \cyclic{X,Y,Z}\kr(X,Y)Z\ =\  \cyclic{X,Y,Z} dV(X,Y)Z.
\edm
yields and interesting consequence: For $\dim M\geq 3$, $X,Y,Z$ can be chosen
linearly independent, hence $dV=0$ and $V$ is locally a gradient field.
Observe that a routine calculation shows that $dV(X,Y)=0$ for all $X$ and 
$Y$ is equivalent to $\langle \nabla^g_X V,Y\rangle = 
\langle \nabla^g_Y V,X\rangle$, and one checks that the same property holds
for $\nabla^g$ replaced by $\nabla$. The triple $(M^n , g , V)$ defines a Weyl
structure, i.e., a conformal class  of Riemannian metrics and a torsion
free connection $\nabla^w$ preserving the conformal class. In general , the Weyl connection
and its curvature tensor are given by the formulas
\begin{eqnarray*}
\nabla^{w}_X Y &=& \nabla^g_X Y\,+\,g \langle X\, ,\, V \rangle 
\, Y\,+\, 
\langle Y \, , \, V \rangle \, X \, - \, \langle X \, , \, Y \rangle 
\, V \, , \\
\mathcal{R}^{\nabla}(X , Y) Z &=& \mathcal{R}^{w}(X,Y)Z \,
- \, d V(X , Y) \, Z \, .
\end{eqnarray*}
The connection $\nabla$ with vectorial torsion is flat if and only if $d V =
0$ and the Weyl connection is flat, $ \mathcal{R}^{w} = 0$.
\begin{prop}
There is a correspondence between triples $(M^n, g, \nabla), \, n \geq 3,$ of Riemannian
manifolds and flat metric connections $\nabla$ with vectorial torsion 
and closed, flat Weyl structures.
\end{prop}

\noindent
In particular, if a Riemannian manifold $(M^n,g)\, , \, n \geq 3,$ admits a flat metric
connection with vectorial torsion, then it is locally conformal flat (the Weyl
tensor vanishes). Moreover, we can apply Theorem 2.1. and Proposition 2.2. of
the paper \cite{Agri&F06}. If $M^n$ is compact, then its universal covering
splits and is conformally equivalent to $S^{n-1} \times \R^1$.\\

\noindent
Let us discuss the exceptional dimension two. In this case, the curvature
$\mathcal{R}^{\nabla}$ is completely defined by {\it one} function, namely
$\langle \mathcal{R}^{\nabla}(e_1, e_2)e_1,e_2 \rangle$. Using the formula 
for the
Riemannian curvature tensor we compute this function and then we obtain immediately
\begin{prop}
Let $(M^2,g)$ be a $2$-dimensional Riemannian manifold with Gaussian curvature
$G$. A metric connection with vectorial torsion is flat if and only if
\bdm
G \ = \ \mathrm{div}^g(V)
\edm 
holds. In particular, if $M^2$ is compact, then $M^2$ diffeomorphic to the
torus or the Klein bottle.
\end{prop}

\section{A family of flat connections on $S^7$}\label{fam-conn}\noindent
%
\subsection{Construction}
In dimension $7$, the complex $\Spin(7)$-representation $\Delta^\C_7$
is the complexification of a real $8$-dimensional representation 
$\kappa: \, \Spin(7)\ra \End(\Delta_7)$, 
since the real Clifford algebra $\mathcal{C}(7)$ is isomorphic to
$\M(8)\oplus \M(8)$. Thus, we may identify  $\R^8$ with the vector space
$\Delta_7$ and embed therein the sphere $S^7$ as the set of all 
spinors of length one. Fix your favorite explicit realization of the
spin representation by skew matrices, 
$\kappa_i:=\kappa(e_i)\in\so(8)\subset \End(\R^8)$, $i=1,\ldots,7$. 
We shall use it to define an explicit parallelization of $S^7$ by
Killing vector fields.
Define  vector fields $V_1,\ldots,V_7$  on $S^7$ by
\bdm
V_i(x)\ =\ \kappa_i \cdot x \text{ for }x\in S^7\subset \Delta_7.
\edm
>From the antisymmetry of $\kappa_1,\ldots,\kappa_7$, 
we easily deduce  the following properties for these vector fields:
\begin{enumerate}
\item They are indeed tangential to $S^7$, $\langle V_i (x),x\rangle = 0$.
\item They are of constant length one,
\bdm
\langle V_i(x),V_i(x)\rangle\, 
=\,  \langle \kappa_i x ,\kappa_i x\rangle
\, =\,  -\langle \kappa_i^2 x , x\rangle\, =\,  -\langle (-1)\cdot x, x\rangle
\ =\ 1.
\edm
\item They are pairwise orthogonal ($i\neq j$).
\end{enumerate}
The commutator of vector fields is inherited from the ambient space,
hence  $[V_i(x),V_j(x)]=[\kappa_i,\kappa_j](x)= 2\kappa_i\kappa_j x$
for $i\neq j$.
In particular, one checks immediately that $[V_i(x),V_j(x)]$ is again
tangential to $S^7$, as it should be. Furthermore,  the
vector fields $V_i(x)$ are Killing.
We now define a connection $\nabla$ on $TS^7$ by $\nabla V_i(x)=0$;
observe that this implies that all tensor fields with constant coefficients
are parallel as well.
This connection  is trivially flat and metric, and its torsion is given by ($i\neq j$)
\bdm
T(V_i,V_j,V_k)(x)\, =\, -\langle [V_i,V_j],V_k\rangle\, =\,
-2\langle \kappa_i \kappa_jx,\kappa_k x\rangle\, =\,
2\langle \kappa_i\kappa_j \kappa_k x,x\rangle.
\edm
If $k$ is equal to $i$ or $j$, this quantity vanishes, otherwise
the laws of Clifford multiplication imply that it is antisymmetric
in all three indices. Observe that this final expression is also valid for
$i = j$, though the intermediate calculation is not. 
Thus, the torsion lies in  $\Lambda^3(S^7)$
as wished, and can be written as
\bdm\tag{$*$}
T(x)\ = \ 2\sum_{i<j<k} \langle \kappa_i\kappa_j \kappa_k x,x \rangle
(V_i\wedge V_j\wedge V_k)(x)\ .
\edm
Since in general $\nabla_X Y = \nabla^g_XY+ T(X,Y,-)/2$, the
definition of $\nabla$ can equivalently be described for the Levi-Civita
connection $\nabla^g$ by
\bdm
\nabla^g_{V_i}V_j\ =\ \left\{\begin{array}{ll} \kappa_i\kappa_j x &
\text{ for }i\neq j \\ 0 & \text{ for }i=j\end{array}\right. .
\edm
$T$ is not $\nabla$-parallel, as it does not have constant
coefficients. Of course, the choice of the vector fields $V_1(x),\ldots,V_7(x)$ is
arbitrary: they can be replaced by any other orthonormal frame
$W_i(x):=A\cdot V_i(x)$ for a transformation $A\in\SO(7)$. However,
any $A\in \mathrm{Stab}\, T \cong G_2\subset \SO(7)$ will yield the same 
torsion and hence connection, thus we obtain a family of 
connections with $7=\dim\SO(7)-\dim G_2$ parameters.
\subsection{$\nabla$ as a $G_2$ connection}
%
The connection $\nabla$ is best understood  from the
point of view of $G_2$ geometry. Recall (see \cite[Thm 4.8]{Friedrich&I1})
that a $7$-dimensional Riemannian manifold $(M^7,g)$ with a fixed $G_2$ 
structure  $\omega\in \Lambda^3(M^7)$ admits a `characteristic' connection 
$\nabla^c$
(i.\,e., a metric $G_2$ connection with antisymmetric torsion) if and only
if it is of  Fernandez-Gray type
$\X_1\oplus\X_3\oplus\X_4$ (see \cite{FG} and \cite[p.~53]{Agricola06} for
this notation).
Furthermore, if existent, $\nabla^c$ is unique, the torsion of $\nabla^c$
is given by 
\be\tag{$**$}
T^c\ = \  -*d\omega-\frac{1}{6}\langle d\omega,*\omega\rangle \omega +
*(\theta\wedge\omega),
\ee
where $\theta$ is the $1$-form that describes
the $\X_4$-component defined by 
$\delta^g(\omega) = - \, (\theta \haken \omega)$.\\

\noindent
Now, \emph{any} generic $3$-form 
$\omega\in\Lambda^3(S^7)$  that is parallel with respect to our connection 
$\nabla$ admits $\nabla$ as its characteristic connection, and is related to
the torsion $T$ given in $(*)$ by the general formula $(**)$.
Thus, there is a large family of $G_2$ structures $\omega$ (namely, all 
generic $3$-forms with constant coefficients) that  induce
the flat connection $\nabla$ as their $G_2$ connection. Let us discuss the possible type of the $G_2$ structures $\omega$
inducing $\nabla$.  
One sees immediately that none of these $G_2$ structures can be nearly 
parallel (type $\X_1$), since $T$ fails to be parallel.
A more elaborate argument shows that they cannot even be
cocalibrated (type $\X_1\oplus \X_3$): by \cite[Thm 5.4]{Friedrich&I1},
a cocalibrated $G_2$ structure on a $7$-dimensional manifold is 
$\mathrm{Ric}^\nabla$-flat if and only if its torsion $T$ is harmonic.
Since $H^3(S^7,\R)=0$, the assertion follows.
Finally, we show that the underlying $G_2$ structures can also not
be locally conformally parallel (type $\X_4$): in \cite[Example
3.1]{Agri&F06}, we showed that such a structure always satisfies
$12\, \delta \theta=6\|T\|^2-\Scal^\nabla$. Since $\nabla$ is flat, the
divergence theorem implies
\bdm
0\ = \ 2\int_{S^7}\delta \theta\,dS^7\ =\ \int_{S^7}\|T\|^2\,dS^7,
\edm
a contradiction to $T\neq 0$. To summarize: There exists a multitude of
$G_2$ structures $\omega\in\Lambda^3(S^7)$ that admit the flat
metric connection $\nabla$ as their characteristic connection; all
these $G_2$ structures are of general type 
$\X_1\oplus \X_3\oplus\X_4$.
%
    
\end{document}